\documentclass{amsart}

\usepackage{amsmath}
\usepackage{amsthm}

\usepackage{graphicx} % Required for inserting images
\usepackage{imakeidx}
\usepackage{hyperref}
\usepackage[nameinlink]{cleveref}
\hypersetup{
    colorlinks = true
}

\usepackage{mathtools}

\usepackage{parskip}
 
\usepackage{cancel}

\usepackage{amsmath, amssymb, amsthm}

\usepackage{tensor}

\usepackage{tikz-cd}

\usepackage{bm}

\usepackage{enumerate}   

\newcommand{\mycomment}[1]{}

\newcommand{\C}{\mathbb{C}}

\newcommand{\inj}{\xhookrightarrow{}}
\renewcommand{\r}{\rightarrow}
\newcommand{\p}{\partial}

\newcommand{\eqn}[1]{\begin{align*}#1\end{align*}}

\newcommand{\ip}[2]{\left\langle{#1},{#2}\right\rangle}

\DeclareMathOperator{\Sym}{Sym}

\DeclareMathOperator{\Ric}{Ric}

\DeclareMathOperator{\sgn}{sgn}

\DeclareMathOperator{\hash}{\sharp}
\DeclareMathOperator{\barhash}{\bar{\sharp}}
\DeclareMathOperator{\hashbarhash}{\sharp\bar{\sharp}}

\usepackage{stackengine}
\usepackage{scalerel}
\usepackage{xcolor}
\usepackage{mathrsfs}

\def\cal{\mathcal}

\def\cR{{\cal R}}

\def\frakR{{\mathfrak R}}

\theoremstyle{plain}
\newtheorem{thm}{Theorem}[section]
\newtheorem{cor}[thm]{Corollary}
\newtheorem{prop}[thm]{Proposition}
\newtheorem{lemma}[thm]{Lemma}

\theoremstyle{remark}
\newtheorem{rem}{Remark}[section]

\DeclareEmphSequence{\bfseries\boldmath}

\makeatletter
\renewenvironment{proof}[1][\proofname]{\par
  \vspace{-\topsep}% remove the space after the theorem
  \pushQED{\qed}%
  \normalfont
  \topsep0pt \partopsep0pt % no space before
  \trivlist
  \item[\hskip\labelsep
        \itshape
    #1\@addpunct{.}]\ignorespaces
}{%
  \popQED\endtrivlist\@endpefalse
  \addvspace{0pt plus 0pt} % some space after
}
\makeatother

\numberwithin{equation}{section}

\title{Kohn--Rossi cohomology and the Bochner technique}

\begin{document}

\begin{abstract}
    We prove a vanishing theorem of Betti numbers on compact, strictly pseudoconvex pseudohermitian manifolds with non-negative curvature operator. The proof is by an application of the Bochner technique to the setting of CR manifolds.
\end{abstract}


\author{Alex Tao}
\address{Department of Mathematics, University of California, Riverside, 900 University Ave, Riverside, CA 92521}
\email{alex.tao@email.ucr.edu}

\maketitle

\section{Introduction}
The literature of curvature conditions imposing topological restrictions dates back to Gauss--Bonnet theorem on compact surfaces. In Riemannian geometry, Bochner made the key observation that the rough Laplacian of a harmonic 1-form can be expressed in terms of the Ricci curvature, which led to the first application of the Bochner technique \cite{bochner46}: if a compact Riemannian manifold has positive Ricci curvature, then the first Betti number vanishes. Meyer and Gallot--Meyer \cite{meyer71,GM75} proceeded to use this technique to produce upper bounds for the Betti numbers of Riemannian manifolds under assumptions on the curvature operator. Based on the conjectures of Nishikawa \cite{Nishikawa86}, recent developments have been made on curvature operators of the second kind \cite{CGT23,NPWW23, FL24}. For a survey of results and generalisations of Bochner technique in Riemannian geometry, one may consult Chapter 9 of \cite{petersen} and \cite{PW21}.

A similar flavour of geometric results are sphere theorems; see \cite{berger03} for a historical account. More recently, the Ricci flow was used by B\"ohm--Wilking \cite{BW08} to prove that manifolds with 2-positive curvature are diffeomorphic to spaceforms, and by Brendle--Schoen \cite{BS08} to prove the quarter-pinched sphere theorem. An analogue of the Riemannian sphere theorems in the K\"ahler setting was conjectured by Frankel \cite{frankel61}: K\"ahler manifolds with positive bisectional curvature are bihomolomorphic to complex projective spaces. The Frankel conjecture was later proved by Mori \cite{Mori79} using algebraic geometry and separately by Siu--Yau \cite{SiuYau80} using harmonic maps. More recently, He--Sun \cite{HS16} proved the Sasakian Frankel conjecture using Ricci flow methods. 

In this paper, we study the Kohn--Rossi cohomology groups in CR geometry, which are analogues of the Dolbeault cohomology groups in complex geometry. Our results extend the Bochner technique to pseudohermitian manifolds and are related to the CR Frankel conjecture; they should be compared to its partial resolution in \cite{Chang14}.  

The main theorem of the paper is the following:
\begin{thm}\label{mainthm}
    Let $(M^{2n+1}, T^{1,0}, \theta)$ be a compact, strictly pseudoconvex pseudohermitian manifold with non-negative pseudohermitian curvature operator $\frakR$. Let $0\leq p\leq n+1$ and $0<q<n$ be integers such that $p+q\notin \{n,n+1\}$. Then every harmonic representative of the $(p,q)$-th Kohn--Rossi cohomological class is parallel and lies in the kernel of $\frakR$. In particular, we have the Betti number bound $b^{p,q}\leq \min\limits_{x\in M}\dim\ker\frakR_x$.
\end{thm}
In the special case of positive curvature operator, we have:
\begin{cor}
    Let $(M^{2n+1}, T^{1,0}, \theta)$ be a compact, strictly pseudoconvex pseudohermitian manifold with non-negative pseudohermitian curvature operator $\frakR$. Let $0\leq p\leq n+1$ and $0<q<n$ be integers such that $p+q\notin \{n,n+1\}$. If $\frakR$ is strictly positive at a point, then $b^{p,q}$ vanishes.
\end{cor}
\begin{proof}
    If $\frakR$ is strictly positive at a point $x\in M$, then $\dim\ker\frakR_x =0$. The conclusion follows immediately from the Betti number bound in Theorem \ref{mainthm}.
\end{proof}
\section{Background}
In this section, we introduce the basic definitions and notation for the paper. One may consult \cite{Takeuchi24} for a brief introduction, or \cite{case} for a comprehensive treatment.
\subsection{CR and Pseudohermitian Manifolds}
A \emph{(nondegenerate) CR manifold} is a pair $(M^{2n+1}, T^{1,0})$ consisting of a real  $(2n+1)$-dimensional manifold $M^{2n+1}$ and a complex rank $n$ distribution $T^{1,0}$ in the complexified tangent bundle $T_\C M$ such that
\begin{itemize}
  \item[(i)] $T^{1,0} \cap T^{0,1} = \{0\}$, where $T^{0,1} := \overline{T^{1,0}}$;
  \item[(ii)] $[T^{1,0}, T^{1,0}] \subset T^{1,0}$; and
  \item[(iii)] if $\theta$ is a real one-form on an open set $U \subset M$ and $H:=\ker \theta = \operatorname{Re}(T^{1,0} \oplus T^{0,1})|_U$, then
  \begin{equation}
    L(Z, W) := -i\, d\theta(Z, \overline{W})
  \end{equation}
  defines a nondegenerate hermitian inner product on $T^{1,0}|_U$.
\end{itemize}
We always assume $\theta$ can be chosen to be globally defined. Given a non-degenerate CR manifold, the 1-form $\theta$ described as above is called the \emph{contact form} and $L$ is called the \emph{Levi form}.

 A \emph{pseudohermitian manifold} is an orientable CR manifold with a distinguished contact form $\theta$. A pseudohermitian manifold is \emph{strictly pseudoconvex} if the Levi form of $\theta$ is positive definite. Given a pseudohermitian manifold, the unique vector field $T$ satisfying $\theta(T)=1$ and $d\theta(T,-)=0$ is called the \emph{Reeb vector field}. Given a local contact form $\theta$, we can extend $\theta$ to a local coframe consisting of a set of complex-valued one-forms $\{\theta, \theta^1,\ldots, \theta^n, \theta^{\bar{1}},\ldots, \theta^{\bar{n}}\}$, where $\theta^{\bar{\beta}}:=\overline{\theta^\beta}$ and each $\theta^{\alpha}$ annihilates the Reeb vector field and the space $T^{0,1}$. A local coframe satisfying these properties is called \emph{admissible}, and with such a choice of coframe, we can write the Levi form $L$ locally as $h_{\alpha\bar{\beta}}$, where
    \[
        d\theta= ih_{\alpha\bar{\beta}}\,\theta^\alpha\wedge\theta^{\bar{\beta}}.
    \]
By non-degeneracy of the Levi form, we can use $h_{\alpha\bar{\beta}}$ to raise and lower indices. 

\subsection{Tanaka--Webster Connection and Curvature}
Let $(M^{2n+1},T^{1,0}, \theta)$ be a pseudohermitian manifold with an admissible coframe $\{\theta, \theta^1,\ldots, \theta^n, \theta^{\bar{1}},\ldots, \theta^{\bar{n}}\}$. Let $\{T, Z_1,\ldots, Z_n, Z_{\bar{1}},\ldots, Z_{\bar{n}}\}$ denote the dual coframe. 
Webster \cite{webster} introduced connection forms $\tensor{\omega}{_\alpha^\beta}$ via the structure equations
    \[
        d\theta^\alpha = \theta^\mu\wedge \tensor{\omega}{_\mu^\alpha}\mod{\theta\wedge\theta^{\bar{\beta}}}
    \]
and defined a unique connection, the \emph{Tanaka--Webster connection} $\nabla$, on $TM\otimes \C$ by
    \[
        \nabla T =0,\qquad \nabla Z_\alpha=\tensor{\omega}{_\alpha^\mu} \otimes Z_\mu, \qquad \nabla Z_{\bar{\beta}} = \tensor{\omega}{_{\bar{\beta}}^{\bar{\nu}}}\otimes Z_{\bar{\nu}},
    \]
where $\tensor{\omega}{_{\bar{\beta}}^{\bar{\nu}}}:=\overline{\tensor{\omega}{_{\beta}^{\nu}}}$. Webster then considered the curvature two-forms
    \[
        \tensor{\Omega}{_\alpha^\gamma} : = d\tensor{\omega}{_\alpha^\gamma}-\tensor{\omega}{_\alpha^\mu}\wedge\tensor{\omega}{_\mu^\gamma},
    \]
and defined functions $R_{\alpha\bar{\beta}\mu\bar{\nu}}$ by
    \[
        \Omega_{\alpha\bar{\beta}} = R_{\alpha\bar{\beta}\mu\bar{\nu}}\,\theta^\mu\wedge\theta^{\bar{\nu}} \mod{\theta\wedge\theta^\alpha, \theta\wedge\theta^{\bar{\beta}}, \theta^\alpha \wedge \theta^\mu, \theta^{\bar{\beta}}\wedge\theta^{\bar{\nu}}}.
    \]
The \emph{pseudohermitian curvature tensor} is given by
    \[
        R_{\alpha\bar{\beta}\mu\bar{\nu}}\,\theta^\alpha\otimes\theta^{\bar{\beta}}\otimes\theta^\mu\otimes\theta^{\bar{\nu}},
    \]
and is analogous to the Riemann curvature tensor for Riemannian manifolds. It satisfies the symmetries
    \begin{equation}  R_{\alpha\bar{\beta}\mu\bar{\nu}}=R_{\alpha\bar{\nu}\mu\bar{\beta}}=R_{\mu\bar{\nu}\alpha\bar{\beta}}.\label{symmetries}
    \end{equation} 
Given a pseudohermitian curvature tensor, one can use the symmetries in \eqref{symmetries} to define an associated hermitian form, the \emph{pseudohermitian curvature operator} $\mathfrak{R}$, on the bundle $\Sym^2((T^{1,0})^*\wedge (T^{0,1})^*)$. Explicitly, we define $\frakR$ to be the map 
    \begin{align}\label{pct}        
        \frakR:\left((T^{1,0})^*\wedge (T^{0,1})^*\right)\otimes \left((T^{1,0})^* \wedge (T^{0,1})^*\right)  &\r \C \nonumber\\  
        \omega\otimes \eta &\mapsto -R_{\alpha\bar{\beta}\mu\bar{\nu}}\omega^{\bar{\beta}\alpha}\bar{\eta}^{\bar{\nu}\mu},
    \end{align}
    where $\bar{\eta}^{\bar{\nu}\mu}= -\overline{\eta^{\bar{\mu}\nu}}$.
The finite dimensional spectral theorem allows us to choose an orthonormal basis consisting of real eigenvectors $\{E^{(i)}\}$ which diagonalizes the operator and yield the formula
    \begin{equation}
        R_{\alpha\bar{\beta}\mu\bar{\nu}} = -\sum_i\lambda_i E^{(i)}_{\alpha\bar{\beta}}E^{(i)}_{\mu\bar{\nu}},\label{diag}
    \end{equation}
    where $\lambda_i$ are the corresponding real eigenvalues of $E^{(i)}$. Applying the symmetries in \eqref{symmetries} to the orthogonal decomposition in \eqref{diag}, we have
        \begin{equation}
             \sum_i\lambda_i E_{\alpha\bar{\beta}}E_{\mu\bar{\nu}}= 
             \sum_i\lambda_i E_{\alpha\bar{\nu}}E_{\mu\bar{\beta}}= 
             \sum_i\lambda_i E_{\mu\bar{\beta}}E_{\alpha\bar{\nu}},\label{diagsymm}
    \end{equation}
    where for readability, we omit the superscript on the eigenvectors and will do so henceforth.

    We define the \emph{pseudohermitian Ricci curvature} to be the contraction of the pseudohermitian curvature tensor in the last two indices, $R_{\alpha\bar{\beta}}\,\theta^\alpha\otimes\theta^{\bar{\beta}}$. And the twice contracted pseudohermitian curvature tensor gives us the \emph{pseudohermitian scalar curvature}, the function $R=\tensor{R}{_\alpha^\alpha}$.
    \begin{rem}\label{signremark}
        One can readily check that $-R_{\alpha\bar{\beta}\mu\bar{\nu}}h^{\alpha\bar{\beta}}\overline{h}^{\mu\bar{\nu}}=R$ so that the choice of negative sign of the curvature operator in \eqref{pct} gives us consistent sign between eigenvalues of $\frakR$ and $R$. For this reason, we shall say that $\mathfrak{R}$ is \emph{non-negative} if $\lambda_i \geq 0$ for every eigenvalue $\lambda_i$ of $\mathfrak{R}$.
    \end{rem}

\subsection{Differential Forms and Kohn--Rossi Cohomology}
We briefly review the formulation in Sections 2.1 and 2.2 of \cite{case} and stress that the Kohn--Rossi complexes are CR invariant. The reader should note that the constructions of the Rumin and the bigraded Rumin complexes by Case deviate from constructions in the literature \cite{Rumin94,GarfieldLee1998}, but the complexes themselves are known to be equivalent.

Let $(M^{2n+1}, T^{1,0})$ be a CR manifold. We write $\Omega^{p,q}$ to denote the space of smooth (global) sections of 
    \[
        \Lambda^{p,q} := \Lambda^p (T^{1,0})^* \otimes \Lambda^q (T^{0,1})^*.
    \]
Given a (global) contact form, there is a canonical embedding $\iota:\Omega^{p,q}\inj\Omega_\C^{p+q}$. Let $(M^{2n+1},T^{1,0},\theta)$ be a pseudohermitian manifold with Reeb vector field $T$. If  $p+q \leq n$, define
    \[
    \mathcal{R}^{p,q} := 
        \left\{ \omega \in\Omega^{p+q}_\C : \theta \wedge \omega \wedge d\theta^{n+1-p-q} = 0, \theta \wedge d\omega \wedge d\theta^{n-p-q} = 0, \omega|_{H_\C}\in \Omega^{p,q} \right\},
    \]
and if $p+q \geq n+1$, define
    \[
        \mathcal{R}^{p,q} :=   
        \left\{ \omega \in \Omega^{p+q}_\C : \theta \wedge \omega = 0, \theta \wedge d\omega = 0, (T\lrcorner~\omega)|_{H_\C}\in \Omega^{p,q} \right\},
    \]
where $H_\C:=H\otimes \C$. For $p+q\leq n$, elements of $\cR^{p,q}$ are trace-free (cf. \cite{case}, Lemma 3.2.8).

By composing exterior derivative $d$ with a natural projection operator on $\cR^{p+q+1}$, one obtains the operators $\bar{\p}_b:\cR^{p,q}\r\cR^{p,q+1}$. The operator $\p_b$ is defined as the conjugate of $\bar{\p}_b$. For each $p\in \{0,\ldots, n+1\}$, we have the \emph{$p$-th Kohn--Rossi complex}
    \[
        0 \longrightarrow \cR^{p,0} \xlongrightarrow{\overline{\p}_b} \cR^{p,1} \xlongrightarrow{\overline{\p}_b} \dots \xlongrightarrow{\overline{\p}_b} \cR^{p,n} \xlongrightarrow{\overline{\p}_b} 0.
    \]
We note again that this definition of the $\overline{\partial}_b$-operator and the Kohn--Rossi complex differs from the original and standard definitions found in the literature \cite{KohnRossi65, FollandKohn72, tanaka, ChenShaw01}; see \cite{case} for details.
Given an integer $q\in [0,n]$, we set 
    \[
        H^{p,q}(M) := \frac{\ker\left(\overline{\partial}_{b} : \mathcal{R}^{p,q} \to \mathcal{R}^{p,q+1}\right)}{\mathrm{im}\left(\overline{\partial}_{b} : \mathcal{R}^{p,q-1} \to \mathcal{R}^{p,q}\right)},
    \]
where $\cR^{p,-1}:=0$ and $\cR^{p,n+1}:=0$. It is known that $H^{p,q}(M)$ is isomorphic to the Kohn--Rossi cohomology groups as defined by Tanaka in \cite{tanaka}. The analogous Hodge decomposition theorem in this setting is proved in Theorem 3.6.3 of \cite{case}, so that every Kohn--Rossi cohomological class can be represented by a unique harmonic form described in Section \ref{formulas}.

\subsection{The Weitzenb\"{o}ck Formula}\label{formulas}
In general, the Bochner technique requires a second order differential operator and a Weitzenb\"{o}ck type formula to obtain a curvature term. The sign of the curvature is then used to control the topology of the underlying manifold by applying Hodge theory. We introduce the notation and related operators to describe harmonic forms in this subsection.

Given an admissible coframe $\{\theta, \theta^1,\ldots, \theta^n, \theta^{\bar{1}},\ldots, \theta^{\bar{n}}\}$, we shall use the multi-index  $A=(\alpha_1,\ldots,\alpha_p)$ and $B=(\beta_1,\ldots, \beta_q)$ to define
\[
    \theta^A :=  \theta^{\alpha_1}\wedge \ldots\wedge \theta^{\alpha_n},\qquad \theta^{\bar{B}} := \theta^{\bar{\beta}_1}\wedge \ldots, \theta^{\bar{\beta}_q}.
\]
We write $A'$ and $B'$ to denote indices of length $p-1$ and $q-1$ respectively. When expressing a contraction $\Lambda^{0,1}\otimes \Lambda^{p,q}\r \Lambda^{p-1,q}$, we will identify $A$ with $(\alpha,A')$ to clearly indicate the contraction of a single index. The use of square brackets on upper (or lower) indices will denote antisymmetrization in the upper (respectively lower) indices of that type only.

Let $(M^{2n+1},T^{1,0}, \theta)$ be a pseudohermitian manifold. The Tanaka--Webster connection induces the following operators:
\begin{align*}
    \nabla_b:\Omega^{p,q}   &\r \Omega^{1,0}M\otimes\Omega^{p,q}M\\
    \nabla_b\omega          &:=\frac{1}{p!q!}\nabla_\gamma\omega_{A\bar{B}}\,\theta^\gamma\otimes\theta^A\otimes\theta^{\bar{B}},\\
    \overline{\nabla}_b:\Omega^{p,q}   &\r \Omega^{0,1}M\otimes\Omega^{p,q}M\\
    \overline{\nabla}_b\omega          &:=\frac{1}{p!q!}\nabla_{\bar{\sigma}}\omega_{A\bar{B}}\,\theta^{\bar{\sigma}}\otimes\theta^A\otimes\theta^{\bar{B}},
\end{align*}
where $\omega = \frac{1}{p!q!}\omega_{A\bar{B}}\,\theta^A\wedge\theta^{\bar{B}}$.

We note that Levi form induces a (pointwise) \emph{Hermitian inner product} on $\Omega^{p,q}$ by
    \[
        \ip{\omega}{\tau}:=\frac{1}{p!q!}\omega_{A\bar{B}}\bar{\tau}^{\bar{B}A},
    \]
for all $\omega = \frac{1}{p!q!}\omega_{A\bar{B}}\,\theta^A\wedge\theta^{\bar{B}}$ and $\tau = \frac{1}{p!q!}\tau_{A\bar{B}}\,\theta^A\wedge\theta^{\bar{B}}$ in $\Omega^{p,q}$, where $\bar{\tau}_{A\bar{B}}:=(-1)^{pq}\overline{\tau_{B\bar{A}}}$. For compactly supported elements $\omega$ and $\tau$ of $\Omega^{p,q}$, we have the induced \emph{$L^2$-inner product}
    \[
        (\omega,\tau):= \frac{1}{n!}\int_M\ip{\omega}{\tau}\,\theta\wedge d\theta^n.
    \]
The formal $L^2$-adjoint of the operators $\p_b,\overline{\p}_b, \nabla_b,\overline{\nabla}_b$ are denoted by $\p_b^*,\overline{\p}_b^*, \nabla_b^*,\overline{\nabla}_b^*$ respectively.

The \emph{Kohn Laplacian} $\Box_b:\cR^{p,q} \r \cR^{p,q}$ is defined by
\eqn{
                \Box_b\omega    :=   \frac{n-p-q}{n-p-q+1}\overline{\p}_b\overline{\p}_b^*+\overline{\p}_b^*\overline{\p}_b
    }
for $p+q\leq n-1$. We say that a $(p,q)$-form $\omega\in \cR^{p,q}$ is \emph{harmonic} if $\Box_b\omega=0$. 

\begin{rem}
    This definition of the Kohn Laplacian is a modification of existing definitions in the literature \cite{FollandKohn72,GarfieldLee1998, tanaka}; nevertheless, it has the same kernel as other definitions and it carries additional desired properties (cf. Section 3.4 \cite{case}).
\end{rem}

The formulas of \cite{case} Proposition 3.4.11 and Corollary 3.4.13 are presented in the following proposition and are essential ingredients in the proof of our results.
\begin{prop}\label{mainprop}
    Let $(M^{2n+1},T^{1,0},\theta)$ be a pseudohermitian manifold. Then
        \begin{align}
            \Box_b\omega &=\frac{q}{n}\nabla_b^*\nabla_b\omega + \frac{n-q}{n}\overline{\nabla}_b^*\overline{\nabla}_b\omega-\frac{1}{n-p-q+1}(\overline{\p}_b\overline{\p}^*_b+\p_b\p^*_b)\omega\label{13.11}\\
            &\quad-R\hashbarhash\omega - \frac{q}{n}\Ric\hash\omega-\frac{n-q}{n}\Ric\barhash\omega\nonumber\\
            &= \frac{(q-1)(n-p-q)}{n(n-p-q+2)}\nabla_b^*\nabla_b\omega + \frac{(n-q+1)(n-p-q)}{n(n-p-q+2)}\overline{\nabla}_b^*\overline{\nabla}_b\omega\\
            &\quad+\frac{1}{n-p-q+2}(\p^*_b\p_b+\overline{\p}^*_b\overline{\p}_b)\omega\nonumber\\
            &\quad-\frac{n-p-q}{n-p-q+2}R\hashbarhash\omega - \frac{(q-1)(n-p-q)}{n(n-p-q+2)}\Ric\hash\omega\nonumber\\
            &\quad-\frac{(n-q+1)(n-p-q)}{n(n-p-q+2)}\Ric\barhash\omega\nonumber 
        \end{align}
        for all $\omega\in \mathcal{R}^{p,q}, p+q\leq n-1$, where
    \begin{align}\label{operators}
        R\hashbarhash\omega :&= \frac{pq}{p!q!}\tensor{R}{_{\alpha\bar{\beta}}^{\bar{\nu}\mu}}\omega_{\mu A'\bar{\nu}\bar{B}'}\,\theta^A\wedge\theta^{\bar{B}},\nonumber\\
        \Ric\hash\omega :&= -\frac{p}{p!q!}\tensor{R}{_{\alpha}^{\mu}}\omega_{\mu A'\bar{B}}\,\theta^A\wedge\theta^{\bar{B}},\\
        \Ric\barhash\omega :&= -\frac{q}{p!q!}\tensor{R}{^{\bar{\nu}}_{\bar{\beta}}}\omega_{A\bar{\nu}\bar{B}'}\,\theta^A\wedge\theta^{\bar{B}}\nonumber.
    \end{align}
\end{prop}

\section{Proof of Theorem 1.1}
We begin with a simple algebraic lemma on antisymmetrization of tensors. 
\begin{lemma}\label{antisymm}
    Let $T_{i_1\cdots i_k}$ be a $(0,k)$-tensor field, $k\geq 1$, then 
        \begin{equation}
            T^{i_1\cdots i_k}T_{[i_1\cdots i_k]}=|T_{[i_1\cdots i_k]}|^2.
        \end{equation}
    Furthermore, if $T_{i_1\cdots i_k}$ is antisymmetric in the last $k-1$ entries, then its antisymmetrization $T_{[i_1\cdots i_k]}$ satisfies
        \begin{equation}\label{antisymmformula}
           kT_{[i_1\cdots i_k]} = T_{i_1\cdots i_k} - T_{i_2i_1i_3\cdots i_k} - \cdots - T_{i_ki_2\cdots i_{k-1}i_1}.  
        \end{equation}
\end{lemma}
\mycomment{
    \begin{proof}
     By the definition of $T_{[i_1\cdots i_k]}$,
        \eqn{
            k!T_{[i_1\cdots i_k]}    &=  \sum_{\sigma\in S_n}\sgn(\sigma)T_{\sigma(i_1)\cdots \sigma(i_k)}\\
            &= \sum_{\sigma\in S_{n-1}}\sgn(\sigma)T_{i_1\sigma(i_2)\cdots \sigma(i_k)}-\cdots - \sum_{\sigma\in S_{n-1}}\sgn(\sigma)T_{i_k\sigma(i_2)\cdots \sigma(i_1)}\\
                                        &= (k-1)!T_{i_1\cdots i_k} - (k-1)!T_{i_2i_1i_3\cdots i_k} - \cdots - (k-1)!T_{i_ki_2i_3\cdots i_{k-1}i_1},
        }
        where the last line uses the fact that antisymmetrization is a projection operator.
    \end{proof}
}

We are now ready to work towards the main theorem.
\begin{lemma}\label{mainlemma}
    Let $(M^{2n+1},T^{1,0},\theta)$ be a strictly pseudoconvex pseudohermitian manifold with non-negative pseudohermitian curvature operator. If $\omega\in \mathcal{R}^{p,q}$ and $N, M_1,$ and $M_2$ are non-negative real numbers such that $N=M_1+M_2$, then 
    \begin{equation}\label{mainineq}
                \ip{- NR\hashbarhash\omega - M_1\Ric\hash\omega - M_2\Ric\barhash\omega}{\omega} \geq0.
    \end{equation}
    Furthermore, equality holds at $x$ if and only if $\mathfrak{R}\omega|_x= 0$.
\end{lemma}
\begin{proof}
Applying the operator definitions in \eqref{operators} and taking inner product determined by the Levi form, the left hand side of \eqref{mainineq} is a positive constant multiple of
    \eqn{
        -Npq\tensor{R}{_{\alpha\bar{\beta}}^{\bar{\nu}\mu}}\omega_{\mu A'\bar{\nu}\bar{B}'}\bar{\omega}^{\bar{\beta}\bar{B}'\alpha A'} + 
        M_1 p\tensor{R}{_{\alpha}^{\mu}}\omega_{\mu A'\bar{B}}\bar{\omega}^{\bar{B}\alpha A'} 
        + M_2q\tensor{R}{^{\bar{\nu}}_{\bar{\beta}}}\omega_{A\bar{\nu}\bar{B}'}\bar{\omega}^{\bar{\beta}\bar{B}'A}.
    }
Assuming non-negative pseudohermitian curvature operator, keeping track of our sign convention as noted in Remark \ref{signremark}, and applying \eqref{diag}, we arrive at 
    \begin{align}
&Npq\sum_i\lambda_iE_{\alpha\bar{\beta}}E^{\bar{\nu}\mu}\omega_{\mu A'\bar{\nu}\bar{B}'}\bar{\omega}^{\bar{\beta}\bar{B}'\alpha A'} - M_1p\sum_i\lambda_iE_{\alpha\bar{\delta}}E^{\bar{\delta}\mu}\omega_{\mu A'\bar{B}}\bar{\omega}^{\bar{B}\alpha A'}\label{maineq}
        \\&\qquad\qquad-M_2q\sum_i\lambda_iE^{\bar{\nu}\delta}E_{\delta\bar{\beta}}\omega_{A\bar{\nu}\bar{B}'}\bar{\omega}^{\bar{\beta}\bar{B}'A}.\nonumber
    \end{align}
    Since each $\lambda_i\geq 0$, it suffices to consider the sign of each term in the summands over $i$. Applying the symmetries of those in \eqref{diagsymm} to the summands in \eqref{maineq}, we have
    \eqn{
        &Npq\tensor{E}{_{\alpha}^\mu}\tensor{E}{^{\bar{\nu}}_{\bar{\beta}}}
        \omega_{\mu A'\bar{\nu}\bar{B}'}\bar{\omega}^{\bar{\beta}\bar{B}'\alpha A'} - 
        M_1p\tensor{E}{^{\bar{\delta}}_{\bar{\delta}}}\tensor{E}{_\alpha^{\mu}}\omega_{\mu A'\bar{B}}\bar{\omega}^{\bar{B}\alpha A'}\\&\qquad\qquad - 
        M_2q\tensor{E}{^{\bar{\nu}}_{\bar{\beta}}}\tensor{E}{_{\delta}^\delta}\omega_{A\bar{\nu}\bar{B}'}\bar{\omega}^{\bar{\beta}\bar{B}'A}.
    }
    Using the splitting $N=M_1+M_2$ and an application of Lemma \ref{antisymm} gives
    \begin{align}
        &Npq\tensor{E}{_{\alpha}^\mu}\tensor{E}{^{\bar{\nu}}_{\bar{\beta}}}
        \omega_{\mu A'\bar{\nu}\bar{B}'}\bar{\omega}^{\bar{\beta}\bar{B}'\alpha A'} - 
        M_1p\tensor{E}{^{\bar{\delta}}_{\bar{\delta}}}\tensor{E}{_\alpha^{\mu}}\omega_{\mu A'\bar{B}}\bar{\omega}^{\bar{B}\alpha A'}\nonumber
        \\&\qquad\qquad - 
        M_2q\tensor{E}{^{\bar{\nu}}_{\bar{\beta}}}\tensor{E}{_{\delta}^\delta}\omega_{A\bar{\nu}\bar{B}'}\bar{\omega}^{\bar{\beta}\bar{B}'A}\nonumber
        \\ = &
        -M_1p\tensor{E}{_\alpha^{\mu}}\omega_{\mu A'\bar{\nu}\bar{B}'}
        (-q\tensor{E}{^{\bar{\nu}}_{\bar{\beta}}}\bar{\omega}^{\bar{\beta}\bar{B}'\alpha A'} + 
        \tensor{E}{^{\bar{\delta}}_{\bar{\delta}}}\bar{\omega}^{\bar{\nu}\bar{B}'\alpha A'})
        \\&\qquad\qquad - 
        M_2q\tensor{E}{^{\bar{\nu}}_{\bar{\beta}}}\bar{\omega}^{\bar{\beta}\bar{B}'\alpha A'}
        (-p\tensor{E}{_{\alpha}^\mu}\omega_{\mu A'\bar{\nu}\bar{B}'} \nonumber + 
        \tensor{E}{_{\delta}^\delta}\omega_{\alpha A'\bar{\nu}\bar{B}'})
        \\ = &\label{lastline}
        M_1(q+1)p\tensor{E}{_\alpha^{\mu}}\omega_{\mu A'\bar{\nu}\bar{B}'}
        (\tensor{E}{^{[\bar{\nu}}_{\bar{\delta}}}\bar{\omega}^{\bar{\delta}\bar{B}']\alpha A'})
        \\&\qquad\qquad + 
        M_2(p+1)q\tensor{E}{^{\bar{\nu}}_{\bar{\beta}}}\bar{\omega}^{\bar{\beta}\bar{B}'\alpha A'}
        (\tensor{E}{_{[\alpha}^\delta}\omega_{\delta A']\bar{\nu}\bar{B}'})\nonumber,
    \end{align}
where in \eqref{lastline} we introduced a negative sign by switching one pair of indices after antisymmetrization. Since $\omega\in \cR^{p,q}$ is trace-free, we can write
\begin{equation}\label{Eq: Trace-free conditions}
    \begin{split}
    \tensor{E}{_\alpha^{\mu}}\omega_{\mu A'\bar{\nu}\bar{B}'} &= 
    \tensor{E}{_{\alpha\bar{\mu}}}\tensor{\omega}{^{\bar{\mu}}_{A'\bar{\nu}\bar{B}'}}=
    (q+1)\tensor{E}{_{\alpha[\bar{\mu}}}\tensor{\omega}{^{\bar{\mu}}_{A'\bar{\nu}\bar{B}']}}
    \\
    \tensor{E}{^{\bar{\nu}}_{\bar{\beta}}}\bar{\omega}^{\bar{\beta}\bar{B}'\alpha A'} &=
    \tensor{E}{^{\bar{\nu}\beta}}\tensor{\bar{\omega}}{_{\beta}^{\bar{B}'\alpha A'}}=
    (p+1)\tensor{E}{^{\bar{\nu}[\beta}}\tensor{\bar{\omega}}{_{\beta}^{\bar{B}'\alpha A']}}.
    \end{split}
\end{equation}
Setting $\tilde{\Omega}:=\tensor{E}{^{[\bar{\nu}}_{\bar{\delta}}}\bar{\omega}^{\bar{\delta}\bar{B}']\alpha A'}$ and $\Omega:=\tensor{E}{_{[\alpha}^\delta}\omega_{\delta A']\bar{\nu}\bar{B}'}$ and using Lemma \ref{antisymm}, the expression \eqref{lastline} simplifies to
    \begin{equation}\label{Eq: positivity}
        M_1(q+1)^2p|\tilde{\Omega}|^2+M_2(p+1)^2q|\Omega|^2,
    \end{equation}    
    which is evidently non-negative. 
    
    For equality to hold in \eqref{mainineq} at a point $x\in M$, \eqref{Eq: positivity} implies $\tilde{\Omega}|_x=\Omega|_x=0$ for every basis vector $E$. In particular, the relations in \eqref{Eq: Trace-free conditions} imply that $\frakR\omega|_x= 0$.
\end{proof}
We can now apply the Bochner technique and obtain our main result.
\begin{proof}[Proof of Theorem \ref{mainthm}]
    Let $(M^{2n+1},T^{1,0},\theta)$ be a pseudohermitian manifold with non-negative pseudohermitian curvature operator. Let $0\leq p\leq n+1$ and $0<q<n$ be integers such that $p+q\notin \{n,n+1\}$. Let $\omega\in \mathcal{R}^{p,q}$ be a harmonic $(p,q)$-form. Then
        \[
            0 = \int_M \ip{\Box_b\omega}{\omega} \theta\wedge d\theta^n.
        \]
    The Hodge star operator induces an isomorphism $H^{p,q}\cong H^{n+1-p,n-q}$ (cf. \cite{case}, Corollary 3.6.8), so it suffices to prove the theorem assuming $p+q\leq n-1$. Using Proposition \ref{mainprop} and Lemma \ref{mainlemma}, we have
        \begin{equation}\label{bochnereq}
            \begin{split}
            0 &= \int_M\left((q-1)|\nabla_b\omega|^2 
            + (n-q+1)|\overline{\nabla}_b\omega|^2\right.
            +\frac{1}{n-p-q}(|\p_b\omega|^2+|\overline{\p}_b\omega|^2)
            \\&\qquad
            +\left\langle-nR\hashbarhash\omega - (q-1)\Ric\hash\omega\left.-(n-q+1)\Ric\barhash\omega,\omega\right\rangle\right)\theta\wedge d\theta^n\\
            &\geq 0.
            \end{split}
        \end{equation}
    It follows immediately that $\overline{\nabla}_b\omega = \p_b\omega=\overline{\p}_b\omega = 0$. We can further deduce that $\nabla_b\omega=0$: if $q\neq 1$, it is clear from \eqref{bochnereq}, and if $q=1$, apply Lemma \ref{mainlemma} to \eqref{13.11} and use the Bochner technique as in \eqref{bochnereq}. We conclude that the inner product in \eqref{bochnereq} is zero and applying Lemma \ref{mainlemma} again yields $\frakR(\omega)=0$.
    
    Lastly, since the harmonic representative $\omega$ being parallel means its values are determined by its behaviour at a single point, applying the Hodge theorem provides the desired bound on the Betti numbers. 
\end{proof}

\section{Acknowledgments}
I would like to thank Jeffrey Case for bringing this topic to my attention, and providing many helpful discussions and comments throughout the preparation of this paper.

%emph setting for bib
\DeclareEmphSequence{\itshape}
\bibliographystyle{amsalpha}
\bibliography{references}
\end{document}